\documentclass[11pt]{article}
\usepackage{amsmath, amssymb, amscd, amsthm, amsfonts,bbm}
\usepackage{mathrsfs}
\usepackage{graphicx}
\usepackage{float} 
\usepackage{xcolor}
\usepackage{subfigure}
\usepackage{hyperref}
\usepackage{enumerate}
\numberwithin{equation}{section}

\title{Backbone probability of planar Brownian motion}

\date{}

\author{Gefei Cai\thanks{Peking University.}\qquad Zhuoyan Xie$^*$}

\newcommand{\R}{\mathbbm{R}}
\newcommand{\A}{\mathbbm{A}}
\newcommand{\D}{\mathbbm{D}}
\newcommand{\C}{\mathbbm{C}}
\newcommand{\cC}{\mathcal{C}}

\newcommand{\sm}{\mathsf{m}}

\newcommand{\E}{\mathbbm{E}}
\newcommand{\1}{\mathbf{1}}
\renewcommand{\P}{\mathbbm{P}}
\renewcommand{\S}{\mathbbm{S}}
\newcommand{\hH}{\mathbbm{H}}

\newcommand{\LF}{\mathrm{LF}}

\newcommand{\CR}{\mathrm{CR}}
\newcommand{\Mod}{\mathrm{Mod}}
\newcommand{\Bac}{\mathsf{Bac}}

\DeclareMathOperator{\SLE}{SLE}

\newcommand{\dist}{\mathrm{dist}}

\newcommand{\lp}{\mathrm{loop}}

\newcommand{\Conf}{\operatorname{Conf}}
\newcommand{\Haar}{\operatorname{Haar}}
\newcommand{\Unif}{\operatorname{Unif}}
\newcommand{\wt}{\widetilde}

\newtheorem{theorem}{Theorem}[section]
\newtheorem{definition}[theorem]{Definition}
\newtheorem{lemma}[theorem]{Lemma}

\newtheorem{corollary}[theorem]{Corollary}

\newtheorem{proposition}[theorem]{Proposition}
\newtheorem{remark}[theorem]{Remark}

\newcommand\ol[1]{\overline{#1}}

\def\@rst #1 #2other{#1}
\newcommand\MR[1]{\relax\ifhmode\unskip\spacefactor3000 \space\fi
  \MRhref{\expandafter\@rst #1 other}{#1}}
\newcommand{\MRhref}[2]{\href{http://www.ams.org/mathscinet-getitem?mr=#1}{MR#2}}

\def\MR#1{\href{http://www.ams.org/mathscinet-getitem?mr=#1}{MR#1}}

\begin{document}

\maketitle

\begin{abstract}
Motivated by critical planar percolation, we investigate a ``backbone'' event of planar Brownian motion, i.e.~the existence of two disjoint subpaths on the Brownian trajectory connecting the $\varepsilon$-neighborhood of the starting point to a macroscopic distance. We show that the probability of~this event is $C(\log|\log\varepsilon|)^{-1}(1+o(1))$ as $\varepsilon\to0$ for some constant $C\in(0,\infty)$.
\end{abstract}

\section{Introduction}
\label{section-intro}
\subsection{Overview and the main result}

There is a strong relation between the planar Brownian motion and critical percolation. In particular, for $n\ge 1$, the Brownian intersection exponent $\zeta_n=\frac{4n^2-1}{12}$~\cite{lsw-bm-exponents1,lsw-bm-exponents2,lsw-bm-exponents3} is the same as the alternating $2n$-arm exponent of critical percolation~\cite{smirnov-werner-percolation}. This can be well understood in the view of conformal restriction~\cite{lsw-restriction}: both of the \emph{hulls} of Brownian motion and percolation cluster satisfy the restriction property, and therefore, their outer boundaries can both be described by Schramm-Loewner evolution (SLE) with parameter $\kappa=\frac{8}{3}$.

Note that in the critical percolation, one can also define the \emph{monochromatic} arm exponents corresponding to that there exist $n$ disjoint arms of the same color joining two boundaries of an annulus. Based on the relation between Brownian motion and critical percolation, it is natural to explore the analog of such monochromatic arm exponents in the context of Brownian motion. Recently,~\cite{nolin2024backboneexponenttwodimensionalpercolation} derives the exact value of percolation monochromatic 2-arm exponent, namely the \emph{backbone exponent}. See references therein for more background on the backbone exponent of percolation. 

In this paper, we investigate the Brownian counterpart, and find that such Brownian backbone probability indeed has an iterated logarithmic decay (so the ``Brownian backbone exponent" is equal to 0). This also shows a big difference between Brownian motion and critical percolation, although their outer boundaries (or hulls) are the same.

To be precise, let $(B_t)_{t\ge0}$ be a planar Brownian motion starting from $0$. Denote $\S^1:=\{z\in\C:|z|=1\}$ and $\D:=\{z\in\C:|z|<1\}$ to be the unit circle and the unit disk, respectively. Let $\tau_\D$ be the first hitting time of $\S^1$ for $(B_t)_{t\ge0}$.
For each $\varepsilon\in(0,\frac{1}{2})$, consider the event
\begin{equation}\label{eq:def-backbone}
\Bac_\varepsilon=\{\exists\ \text{two disjoint subpaths on the trajectory } B[0,\tau_{\D}] \text{ joining } \varepsilon\S^1 \text{ and } \frac{1}{2}\S^1\}.
\end{equation}
Namely, $\Bac_\varepsilon$ happens if there are two continuous curves $\gamma^1,\gamma^2:[0,1]\to\C$ such that $\gamma^1[0,1]\cap\gamma^2[0,1]=\emptyset$, $\gamma^i[0,1]\subset B[0,\tau_{\D}]$, $\gamma^i(0)\in\varepsilon\S^1$, and $\gamma^i(1)\in\frac{1}{2}\S^1$ for $i=1,2$. Using terminology from percolation, we call $\Bac_\varepsilon$ the backbone event for the Brownian trajectory $B[0,\tau_\D]$.

The main result in this paper is the following iterated logarithmic decay for the probability $\P[\Bac_\varepsilon]$.

\begin{theorem}\label{thm:backbone}
There exists a constant $C\in(0,\infty)$ such that,
\begin{equation}\label{eq:bb-proba}
\lim\limits_{\varepsilon\to0}\P[\Bac_\varepsilon]\cdot(\log|\log\varepsilon|)=C.
\end{equation}
\end{theorem}
In fact, our proof gives an explicit expression of the constant $C$ in~\eqref{eq:bb-proba} via the quantity $S$ defined in~\eqref{eq:s}; see~\eqref{eq:c-s} in Section~\ref{section-backbone-sharp}.
\begin{remark}
The same result also holds when $\frac{1}{2}\S^1$ in~\eqref{eq:def-backbone} is changed to any fixed $r\S^1$ with $r\in(0,1)$ (and for sufficiently small $\varepsilon$), except that the corresponding constant in~\eqref{eq:bb-proba} will depend on $r$.
\end{remark}

Note that the backbone event $\Bac_\varepsilon$ is closely related to the spatial distribution of the \emph{cut points} of $B[0,\tau_\D]$. Indeed, our proof of Theorem~\ref{thm:backbone} is based on our recent paper with Fu and Sun~\cite{cfsx}, from which we can define a layer structure for the cut points of $B[0,\tau_\D]$ and then solve it explicitly; see Proposition~\ref{prop-cr-rho} in Section~\ref{section-backbone-1}. 
The remaining proof of Theorem~\ref{thm:backbone} is then in a similar fashion to~\cite[Lemma 2.6]{jego2023crossing}; see Section~\ref{section-backbone-sharp}.
We mention that~\cite{cfsx} heavily relies on the connection between planar Brownian motion and $\SLE_{8/3}$, and especially their coupling with Liouville quantum gravity (LQG). It would be interesting to find a derivation of Theorem~\ref{thm:backbone} without relying on LQG.

Theorem~\ref{thm:backbone} is related to a certain kind of special points of Brownian motion. Namely, let
\[
\mathcal{B}:=\{B_s:s\in[0,\tau_\D) \text{ and there exists } \varepsilon>0 \text{ such that } (B_{s+u})_{0\le u\le\varepsilon} \text{ does not have a cut point}\}.
\]
Then Theorem~\ref{thm:backbone} indicates that $\mathcal{B}$ is non-empty and has Hausdorff dimension 2. Furthermore, it also suggests that the Hausdorff measure of $\mathcal{B}$ with the gauge $r\mapsto r^2\log\frac{1}{r}\log\log\frac{1}{r}\log\log\log\frac{1}{r}$ would exist and be non-trivial (since the Hausdorff gauge of $B[0,\tau_\D]$ itself is $r\mapsto r^2\log\frac{1}{r}\log\log\log\frac{1}{r}$, see e.g.~\cite{LG85}). In particular, due to that planar Brownian motion a.s.~has no double cut points~\cite{BL}, we know that $\mathcal{B}$ contains the set of double points of $(B_t)_{0\le t\le\tau_\D}$, which has the Hausdorff gauge $r\mapsto r^2(\log\frac{1}{r}\log\log\log\frac{1}{r})^2$~\cite{LG87}.

\subsection{Outlook and discussions}
\label{section-discussion}

Here we give several remarks and related questions before going into the proof.

\begin{itemize}
\setlength{\itemsep}{0pt}
\setlength{\parskip}{0pt}
\setlength{\parsep}{0pt}
\item One can similarly consider the probability that there exist $2n$ disjoint subpaths on $B[0,\tau_\D]$ joining $\varepsilon\S^1$ and $\frac{1}{2}\S^1$ with $n\ge2$. Note that for $n\ge2$, such $2n$-arm probability is not straightforwardly related to the cut points of $B[0,\tau_\D]$. Instead, one need to consider the $(2n-1)$-tuples of the local cut points such that removing these $(2n-1)$ cut points from the trajectory $B[0,\tau_\D]$, $0$ and $B_{\tau_\D}$ are not in the same connected component of the remaining set. However, analyzing such tuples of local cut points (e.g.~the counterpart of the layer structure in Section~\ref{section-backbone-1}) becomes much more complicated, and it seems difficult to solve them explicitly. We mention that similar difficulty also appears in deriving the monochromatic $k$-arm exponents of critical planar percolation for $k\ge3$, see~\cite[Remark 2.3]{nolin2024backboneexponenttwodimensionalpercolation}.

\item 
There is also a natural half-plane variant of our setup. Namely, let $(e_t)_{t\ge0}$ be the Brownian excursion on the upper half plane $\hH$ from $0$ to $\infty$, and define $R\S^+:=R\S^1\cap\hH$ for $R>0$. Then for $n\ge1$, consider the asymptotic probability that 
there exists $(2n+1)$ disjoint subpaths on the trajectory of $(e_t)$ joining $\S^+$ and $R\S^+$ as $R\to\infty$.
Note that the case $n=1$ can be similarly related to the cut points of $(e_t)$, which has been proven in~\cite[Theorem 4]{brownian-beads}.

\item One can also consider the backbone probability for three-dimensional Brownian motion, i.e.~the probability that there are two disjoint subpaths joining $\varepsilon\S^2$ and $\frac{1}{2}\S^2$ on the Brownian tracjectory starting from $0$ until hitting $\S^2$. We conjecture that such probability equals $\varepsilon^{\alpha+o(1)}$ as $\varepsilon\to0$ for some $\alpha\in(0,\infty)$. (The backbone probability for $d\ge4$ dimensional Brownian motion becomes less interesting since the Brownian motion is now a simple curve.)

\item In our forthcoming work~\cite{cx}, we will extend the result of this paper to the Brownian loop soup cases. Let $(B_t)_{0\le t\le\tau_\D}$ be as before, and let $\mathcal{L}$ be an independent Brownian loop soup on $\D$ with intensity $\frac{c}{2}$ for $c\in(0,1]$. 
For a subset $A\subset\ol\D$, denote $\mathcal{C}(A)$ to be the union of $A$ and all loop-soup clusters in $\mathcal{L}$ intersecting with $A$, 
and let $H:=\mathcal{C}(B[0,\tau_\D])$. 
However, unlike the Brownian motion case, 
there is a positive probability that $0$ itself can have two disjoint subpaths (except on $0$) on $H$
connected to $\frac{1}{2}\S^1$. Moreover, in~\cite{cx} we will show the following
\begin{theorem}
We say $t\in(0,\tau_\D]$ is a cut time of $H$ if $\mathcal{C}(B[0,t])\cap \mathcal{C}(B[t,\tau_\D])=B_t$. Let $t_1$ be the largest cut time of $H$ such that the boundary of $\mathcal{C}(B[0,t_1])$ is a simple loop, and denote $E$ to be the event that $t_1$ is the smallest cut time of $H$. 
Then for $c\in(0,1)$,
\begin{equation}\label{eq:p_e}
\P[E]=2^{\frac{c+1}{2}}\sqrt{\frac{6}{1-c}}\left(\int_0^\infty\tau^{-1-\frac{c}{2}}e^{\frac{1-c}{6}\pi\tau}\eta(2i\tau)^{1-c}d\tau\right)^{-1}
\end{equation}
where $\eta(2i\tau):=e^{-\frac{\pi}{6}\tau}\prod_{n=1}^{\infty}\left(1-e^{-4\pi n\tau}\right)$ is the Dedekind eta function.
\end{theorem}
Note that the right side of~\eqref{eq:p_e} tends to $1$ as $c\uparrow1$, and tends to $0$ as $c\downarrow0$.
\end{itemize}

\medskip
\noindent\textbf{Organization of the paper.} In Section~\ref{section-SLE loop}, we first review the explicit relation between the laws of conformal radii and moduli for $\SLE_{8/3}$-type loops, which is an important ingredient of Theorem~\ref{thm:backbone}. Such relation was implicitly established in~\cite{ARS-Annulus}, and we will provide its proof in Appendix~\ref{appendix} for completeness. Then Sections~\ref{section-backbone-1} and~\ref{section-backbone-sharp} are devoted to prove Theorem~\ref{thm:backbone}. 

\medskip
\noindent\textbf{Basic notations.}
For two compact sets $A,B\subset\C$, let $\dist(A,B):=\inf\{|a-b|:a\in A,b\in B\}$.
For a simple loop $\ell\subset\C$, let $D(\ell)$ be the bounded connected component of $\C\backslash\ell$. For a simple loop $\ell\subset\C$ with $0\in D(\ell)$, we denote the conformal radius of $D(\ell)$ seen from 0 by $\CR(\ell,0)$, i.e. if $f:\D\to D(\ell)$ is a conformal map that fixes the origin, then $\CR(\ell,0)=|f'(0)|$.

For $0<r<1$, let $\A_r$ be the standard annulus $\{z\in\C: r<|z|<1\}$. For each annular domain $A\subset\C$,
there exists a unique $\tau>0$ such that $A$ and $\A_{e^{-2\pi\tau}}$ are conformally equivalent. We call $\tau$ the \textit{modulus} of $A$ and write $\Mod(A):=\tau$. For two simple loops $\eta_1,\eta_2$ such that $\overline{D(\eta_1)}\subset D(\eta_2)$, we also write the modulus of the annular domain $D(\eta_2)\setminus\overline{D(\eta_1)}$ as $\Mod(\eta_1,\eta_2)$ for simplicity.

We will frequently deal with elementary functions such as $\sqrt{x},\sinh(\sqrt{x}),\cosh(\sqrt{x})$, and $\tanh(\sqrt{x})$. In the following, we view them as functions defined on $\R$ by taking $i\sqrt{|x|},i\sin(\sqrt{|x|}),\cos(\sqrt{|x|})$, and $i\tan(\sqrt{|x|})$ for $x<0$, respectively.

\medskip
\noindent\textbf{Acknowledgment.}
We are grateful to Xin Sun for many helpful comments and suggestions on the early draft of this paper. We also thank Xinyi Li for helpful discussions. G.C. and Z.X.\ were partially supported by National Key R\&D Program of China (No.\ 2023YFA1010700). G.C. was partially supported by National Key R\&D Program of China (No. 2021YFA1002700).

\section{Conformal radii of $\SLE_{8/3}$-type loops}
\label{section-SLE loop}
The $\SLE_{8/3}$ loop measure on $\C$ is a canonical infinite measure on simple loops characterized by conformal restriction~\cite{werner-loops}, which can be obtained by taking outer boundaries from the Brownian loop measure on $\C$. According to~\cite{gwynne-miller-gluing,loopwelding}, the $\SLE_{8/3}$ loop measure on $\C$ also describes the welding interface when conformally welding two independent Brownian disks into a Brownian sphere.

Let $\SLE^{\lp}_{8/3,\D}$ be the $\SLE_{8/3}$ loop measure on $\C$ restricted to the loops that are contained in $\D$ and surround 0. 
By~\cite{ARS-Annulus}, $\SLE_{8/3,\D}^{\lp}$ is the welding interface when comformally welding a Brownian disk and an independent Brownian annulus into a larger Brownian disk. This, combined with integrability of Liouville conformal field theory on the annulus by~\cite{wu-annulus}, gives the following

\begin{theorem}[\cite{ARS-Annulus}]
\label{thm-cr-mod-relation}
Let $\mu_{\D}$ be a measure on simple loops contained in $\D$ and surrounding 0, such that $\frac{d\mu_{\D}}{d\SLE^{\lp}_{8/3,\D}}(\eta)=f(\Mod(\eta,\S^1))$ for some positive continuous function $f:\R_+\to\R_+$.
Then there is a constant $C>0$ such that for each $\lambda\ge0$,
$$
\int\CR(\eta,0)^\lambda\mu_{\D}(d\eta)=C\frac{\sqrt{12\lambda-1}}{\sinh\left(\frac{\pi}{3}\sqrt{12\lambda-1}\right)}\int_0^\infty e^{-(2\lambda-\frac{1}{6})\pi\tau}\eta(2i\tau)f(\tau)d \tau.
$$
Here $\eta(2i\tau)=e^{-\frac{\pi}{6}\tau}\prod_{n=1}^\infty(1-e^{-4\pi n\tau})$ is the Dedekind eta function.
When $\lambda\in(0,\frac{1}{12})$, the right side is defined by analytic continuation; see the end of Section~\ref{section-intro}.
\end{theorem}

The proof of Theorem~\ref{thm-cr-mod-relation} is implicit in~\cite{ARS-Annulus}, and we will provide a proof in Appendix~\ref{appendix} for completeness. In the following, we turn to the proof of Theorem~\ref{thm:backbone}.

\section{Exact laws of the conformal radii in $B[0,\tau_\D]$}
\label{section-backbone-1}
For the Brownian trajectory $B[0,\tau_{\D}]$, we say $t\in (0,\tau_{\D}) $
is a \emph{cut time} if $B[0,t)\cap B(t,\tau_{\D}]=\emptyset$, and the corresponding 
$B_t$ is called a cut point.
Note that for any $\varepsilon>0$, $B[0,\tau_\D]$ has infinitely many cut points contained in $\varepsilon\D$~\cite[Theorem 2.2]{burdzy-cut-points}.

\begin{definition}\label{def:cut-point}
Let $(t_i)_{i\ge1}$ be the collection of cut time $t$'s for $B[0,\tau_\D]$ such that the outer boundary of $B[0,t]$ is a simple loop, ranked in decreasing order.
Note that a.s.~$\lim_{i\to\infty}t_i=0$.
For each $i\ge1$, define $s_i$ to be the last cut time before $t_i$ (thus $t_i>s_i>t_{i+1}$). Let $\ell_i$ and $\wt\ell_i$ be the outer boundary of $B[0,t_i]$ and the boundary of the connected component of $\D\setminus B[s_i,t_i]$ containing the origin, respectively.
\end{definition}

\begin{figure}[htbp]
    \centering
    \includegraphics[width=0.4\linewidth]{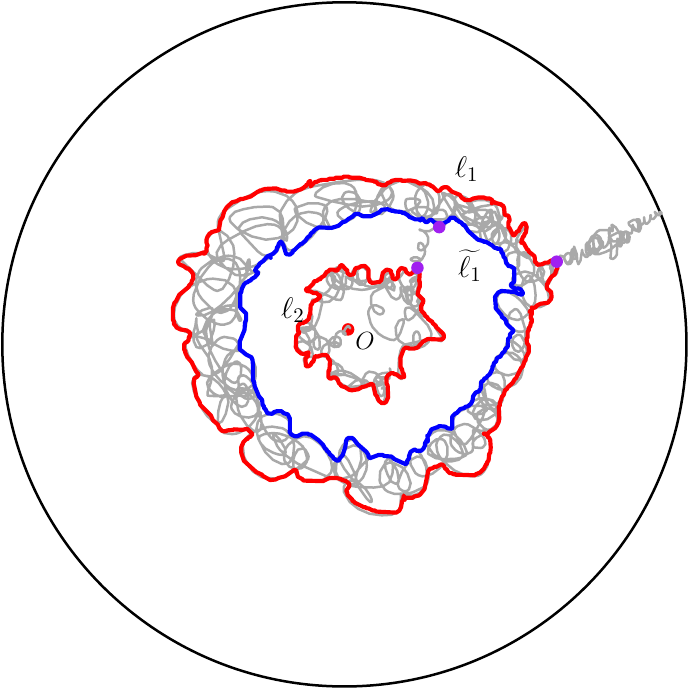}
    \caption{Illustration of Definition~\ref{def:cut-point}.}
\end{figure}

Let $f$ be the conformal map from $D(\ell_1)$ to $\D$ with $f(0)=0$, $f'(0)>0$ and let $\rho=f(\wt\ell_1)$. Note that $\CR(\rho,0)=\frac{\CR(\wt\ell_1,0)}{\CR(\ell_1,0)}$. 
The main result of this section is the following exact law of $\CR(\rho,0)$, which is crucial to the final proof of Theorem~\ref{thm:backbone} in Section~\ref{section-backbone-sharp}.
\begin{proposition}
\label{prop-cr-rho}
For $\lambda\ge0$, we have
\begin{equation}\label{eq:cr-rho}
\E[\CR(\rho,0)^\lambda]=\frac{\sinh(\frac{\pi}{2}\sqrt{12\lambda-1})}{\sqrt{12\lambda-1}}-\frac{2\sqrt{3}\sinh(\frac{\pi}{3}\sqrt{12\lambda-1})}{\sqrt{12\lambda-1}\log\left(\frac{(2+\sqrt{3})^{2}+\tanh^{2}\left(\frac{\pi}{12}\sqrt{12\lambda-1}\right)}{1+(2+\sqrt{3})^{2}\tanh^{2}\left(\frac{\pi}{12}\sqrt{12\lambda-1}\right)}\right)}.
\end{equation}
When $\lambda\in[0,\frac{1}{12})$, the right side of~\eqref{eq:cr-rho} is defined by analytic continuation; see the end of Section~\ref{section-intro}.
\end{proposition}

The proof of Proposition~\ref{prop-cr-rho} relies on the exact solvability of $\CR(\ell_i,0)$ for each $i\ge1$ in Definition~\ref{def:cut-point}. The following explicit law of $\CR(\ell_1,0)$ is originally derived by~\cite[Theorem 1.5]{qian2021generalized} using the radial hypergeometric SLE; see~\cite[Remark 4.9]{cfsx} for another derivation using LQG.
\begin{lemma}
    \label{lem-CR-BM-ell1}
    For $\lambda\ge0$, we have
    \begin{equation}
    \label{eq-CR-BM-ell1}
        \E[\CR(\ell_1,0)^\lambda]=\frac{\sqrt{12\lambda-1}}{\sinh(\frac{\pi}{2}\sqrt{12\lambda-1})}.
    \end{equation}
    When $\lambda\in[0,\frac{1}{12})$, the right side of~\eqref{eq-CR-BM-ell1} is defined by analytic continuation. In particular, $\P[\CR(\ell_1,0)<\varepsilon]=\frac{16}{3\pi}\varepsilon^{\frac{1}{4}}(1+o(1))$ as $\varepsilon\downarrow0$.
\end{lemma}

Let $\nu_\D$ be the law of the loop chosen from the counting measure over $\{\ell_i\}_{i\ge1}$; namely, $\nu_\D$ is such that $\int F(\eta)\nu_\D(d\eta)=\P\left[\sum_{i\ge1}F(\ell_i)\right]$ for any bounded measurable function $F$. The following result, based on~\cite{cfsx} and Theorem~\ref{thm-cr-mod-relation}, gives the explicit law of the conformal radius under $\nu_\D$.
\begin{lemma}\label{lem-CR-BM-counting-ell}
There is a constant $C>0$ such that for $\lambda\ge0$, we have
\begin{equation}
    \label{eq-CR-BM-counting-ell}
        \int\CR(\eta,0)^\lambda\nu_{\D}(d\eta)=C\frac{\sqrt{12\lambda-1}}{\sinh(\frac{\pi}{3}\pi\sqrt{12\lambda-1})}\log\left(\frac{(2+\sqrt{3})^{2}+\tanh^{2}\left(\frac{\pi}{12}\sqrt{12\lambda-1}\right)}{1+(2+\sqrt{3})^{2}\tanh^{2}\left(\frac{\pi}{12}\sqrt{12\lambda-1}\right)}\right).
    \end{equation}
When $\lambda\in(0,\frac{1}{12})$, the right side of~\eqref{eq-CR-BM-counting-ell} is defined via analytic continuation as before.
\end{lemma}
\begin{proof}
By combining~\cite[Lemma 5.3]{cfsx} with~\cite[Theorem 1.3]{cfsx}, we have $\frac{d\nu_{\D}}{d\SLE^{\lp}_{8/3,\D}}(\eta)=\frac{C_1}{\Mod(\eta,\S^1)}$ for some constant $C_1>0$.
Then by Theorem~\ref{thm-cr-mod-relation}, we find
\begin{equation}\label{eq:cr-integral}
\int\CR(\eta,0)^\lambda\nu_{\D}(d\eta)=C
    \frac{\sqrt{12\lambda-1}}{\sinh(\frac{\pi}{3}\sqrt{12\lambda-1})}\int_0^\infty e^{-(2\lambda-\frac{1}{6})\pi\tau}\frac{\eta(2i\tau)}{\tau}d\tau.
\end{equation}
for some constant $C>0$. 
Note that for $a>-\frac{\pi}{6}$, we have $\int_0^\infty e^{-a\tau}\eta(2i\tau)d\tau=\sqrt{\frac{\pi}{2a}}\frac{\sinh\left(\sqrt{\frac{2}{3}\pi a}\right)}{\cosh\left(\sqrt{\frac{3}{2}\pi a}\right)}$ (see e.g.~\cite[Equation (A.4)]{ARS-Annulus}) and
    $$
    \frac{d}{da}\left[-\log\left(\frac{(2+\sqrt{3})^{2}+\tanh^{2}\left(\frac{1}{2}\sqrt{\frac{\pi}{6}a}\right)}{(2-\sqrt{3})^2+\tanh^{2}\left(\frac{1}{2}\sqrt{\frac{\pi}{6}a}\right)}\right)\right]=\sqrt{\frac{\pi}{2a}}\frac{\sinh\left(\sqrt{\frac{2}{3}\pi a}\right)}{\cosh\left(\sqrt{\frac{3}{2}\pi a}\right)}.
    $$
Therefore, we obtain
    $$
    \int_0^\infty e^{-a\tau}\frac{\eta(2i\tau)}{\tau}d\tau=\log\left(\frac{(2+\sqrt{3})^{2}+\tanh^{2}\left(\frac{1}{2}\sqrt{\frac{\pi}{6}a}\right)}{1+(2+\sqrt{3})^{2}\tanh^{2}\left(\frac{1}{2}\sqrt{\frac{\pi}{6}a}\right)}\right).
    $$
for $a>-\frac{\pi}{6}$. Combined with the right side of~\eqref{eq:cr-integral}, we conclude.
\end{proof}

To obtain Proposition~\ref{prop-cr-rho} from Lemmas~\ref{lem-CR-BM-ell1} and~\ref{lem-CR-BM-counting-ell}, we also need the following i.i.d.~structure of $\{\ell_i\}_{i\ge1}$, which is a radial variant of the Brownian beads decomposition established in~\cite{brownian-beads}.

\begin{lemma}\label{lem-independent-decomposition-BM}
\label{lem-independent-BM-CR}
Set $\wt\ell_0=\S^1$ and $s_0=0$. For $n\ge1$, we have:
\begin{enumerate}[(i)]
\item Let $g_n:D(\ell_n)\to\D$ be the conformal map such that $g_n(0)=0$ and $g_n(B_{t_n})=1$. Then $g_n(\wt\ell_n)$ is independent of $\ell_n$ and $\wt\ell_{n-1}$. In particular, $\frac{\CR(\wt\ell_n,0)}{\CR(\ell_n,0)}$ is independent of $\ell_n$ and $\wt\ell_{n-1}$.
\item Let $\wt g_n:D(\wt\ell_n)\to\D$ be the conformal map such that $\wt g_n(0)=0$ and $\wt g_n(B_{s_n})=1$. Then $(\wt g_n(B_t))_{t\in[0,s_n]}$ is independent of $(B_t)_{t\in[s_n,\tau_\D]}$, and has the same law (up to a time change) as $(B_t)_{t\in[0,\tau_\D]}$ conditioned on $B_{\tau_{\D}}=1$.
\end{enumerate}
As a corollary, $\left(\frac{\CR(\wt\ell_n,0)}{\CR(\ell_n,0)}\right)_{n\ge1}$ and $\left(\frac{\CR(\ell_{n+1},0)}{\CR(\wt\ell_n,0)}\right)_{n\ge0}$ are two i.i.d.~r.v.~sequences, and mutually independent of each other.
\end{lemma}
\begin{proof}
    We first prove (ii). Due to rotational invariance, it suffices to consider the case when conditioning on $B_{\tau_\D}=1$.
    Let $(Y_t)_{t\ge0}$ be the time reversal of $(B_t)_{t\in[0,\tau_\D]}$ conditioned on $B_{\tau_{\D}}=1$. Note that $Y_t$ is indeed a Brownian excursion in $\D$ from 1 to 0. 
    Let $(\mathcal{F}_t)_{t\ge0}$ be the natural filtration of $(Y_t)_{t\ge0}$, and $(\mathcal{G}_t)_{t\ge0}$ be the cut time filtration defined in~\cite[Section 3]{brownian-beads} (i.e.~$\mathcal{G}_t$ is the $\sigma$-field generated by $\mathcal{F}_t$ and the set of all cut times of $(Y_t)_{t\ge0}$ before time $t$).
    By~\cite[Proposition 14]{brownian-beads}, if $\tau$ is a $(\mathcal{G}_t)_{t\ge0}$-stopping time such that $\tau$ is also a cut time of $(Y_t)_{t\ge0}$, then the conditional law of $(Y_{\tau+u})_{u \ge0}$ given $\mathcal{G}_\tau$ is a Brownian excursion from $Y_\tau$ to 0 in the connected component of $\D\backslash Y[0,\tau]$ containing $0$. Now we consider $s_n':=\tau_\D-s_n$. Note that
    $s_n'$ is the first cut time of $(Y_t)_{t\ge0}$ after $s_{n-1}'$ such that $Y[s'_{n-1},s'_n]$ disconnects the origin from $\S^1$, thus is a $(\mathcal{G}_t)_{t\ge0}$-stopping time. Therefore, the conditional law of $(\wt g_n(Y_{s_n'+u}))_{u\ge0}$ given $\mathcal{G}_{s_n'}$ is an independent Brownian excursion (up to a time change) in $\D$ from $1$ to $0$. By taking the time reversal of $(Y_t)_{t\ge0}$, we obtain the result.

    Then we prove (i). 
    According to (ii), $(\wt g_{n-1}(B_t))_{t\in[0,s_{n-1}]}$ is independent of $(\wt\ell_{n-1},B_{s_{n-1}})$, and has the same law (up to a time change) as $(B_t)_{t\in[0,\tau_\D]}$ conditioned on $B_{\tau_{\D}}=1$. Since $(g_n(\wt\ell_n),\wt g_{n-1}(\ell_n))$ is determined by $(\wt g_{n-1}(B_t))_{t\in[0,s_{n-1}]}$, it follows that $(g_n(\wt\ell_n),\wt g_{n-1}(\ell_n))$ is independent of $(\wt\ell_{n-1},B_{s_{n-1}})$.

    Now we show that $g_n(\wt\ell_n)$ and $\wt g_{n-1}(\ell_n)$ are independent.
    Let $h_n:D(\wt g_{n-1}(\ell_n))\to\D$ be the conformal map with $h_n(0)=0$ and $h_n(\wt g_{n-1}(B_{t_n}))=1$ (hence  $g_n=h_n\circ(\wt g_{n-1}|_{D(\ell_n)})$).
    Let $A\subset\ol\D$ be relatively compact such that $\overline{\D\backslash A}$ is simply connected and contains $0,1$. By the conformal restriction property of Brownian excursion, the law of $(\wt g_{n-1}(B_t))_{t\in[0,s_{n-1}]}$ conditioned on the event $\wt g_{n-1}(B[0,s_{n-1}])\cap A=\emptyset$ is the same as the Brownian motion (up to a time change) in $\D\backslash A$ starting from $0$ and conditioned on exiting $\D\backslash A$ at $1$. Combined with the conformal invariance of the Brownian motion, it implies that $h_n\circ\wt g_{n-1}(\wt\ell_n)$ is independent of the event $\wt g_{n-1}(B[0,s_{n-1}])\cap A=\emptyset$. By varying $A$, we obtain that $g_n(\wt\ell_n)$ is independent of the outer boundary of $\wt g_{n-1}(B[0,s_{n-1}])$, and hence $\wt g_{n-1}(\ell_n)$.
    
    Combining the previous two paragraphs, we see that $g_n(\wt\ell_n)$, $\wt g_{n-1}(\ell_n)$, $(\wt\ell_{n-1},B_{s_{n-1}})$ are mutually independent. Since $\wt g_{n-1}$ is determined by $\wt\ell_{n-1}$ and $B_{s_{n-1}}$, we have $\sigma(\wt g_{n-1}(\ell_n),\wt\ell_{n-1},B_{s_{n-1}})$ equals $\sigma(\ell_n,\wt\ell_{n-1},B_{s_{n-1}})$. Therefore, $g_n(\wt\ell_n)$ is independent of the triple $(\ell_n,\wt\ell_{n-1},B_{s_{n-1}})$, as desired.
\end{proof}

Now Proposition~\ref{prop-cr-rho} follows by combining Lemma~\ref{lem-independent-BM-CR} and Lemmas~\ref{lem-CR-BM-ell1} and~\ref{lem-CR-BM-counting-ell} above.
\begin{proof}[Proof of Proposition~\ref{prop-cr-rho}]
By Lemma~\ref{lem-independent-BM-CR} and the definition of $\nu_{\D}$, we have
$$
\begin{aligned}
    \int\CR(\eta,0)^\lambda\nu_{\D}(d\eta)&=\sum\limits_{n\ge1}\E[\CR(\ell_n,0)^\lambda]
    =\sum\limits_{n\ge1}\E\left[\left(\prod\limits_{k=2}^n\frac{\CR(\ell_k,0)}{\CR(\wt\ell_{k-1},0)}\cdot\frac{\CR(\wt\ell_{k-1},0)}{\CR(\ell_{k-1},0)}\right)^\lambda\CR(\ell_1,0)^\lambda\right]\\
    &=\sum\limits_{n\ge1}\E[\CR(\ell_1,0)^\lambda]^n\E[\CR(\rho,0)^\lambda]^{n-1}=\frac{\E[\CR(\ell_1,0)^\lambda]}{1-\E[\CR(\ell_1,0)^\lambda]\E[\CR(\rho,0)^\lambda]}.
\end{aligned}
$$
Thus, combined with Lemmas~\ref{lem-CR-BM-ell1} and~\ref{lem-CR-BM-counting-ell}, we obtain
\begin{align}
    \E[\CR(\rho,0)^\lambda]&=\frac{1}{\E[\CR(\ell_1,0)^\lambda]}-\frac{1}{\nu_{\D}[\CR(\eta,0)^\lambda]}\nonumber\\
    &=\frac{\sinh(\frac{\pi}{2}\sqrt{12\lambda-1})}{\sqrt{12\lambda-1}}-\frac{\sinh(\frac{\pi}{3}\sqrt{12\lambda-1})}{C\sqrt{12\lambda-1}\log\left(\frac{(2+\sqrt{3})^{2}+\tanh^{2}\left(\frac{\pi}{12}\sqrt{12\lambda-1}\right)}{1+(2+\sqrt{3})^{2}\tanh^{2}\left(\frac{\pi}{12}\sqrt{12\lambda-1}\right)}\right)}.\label{eq:cr-rho-c}
\end{align}
Here the constant $C$ is the same as the one in Lemma~\ref{lem-CR-BM-counting-ell}.
Now, note that~\eqref{eq:cr-rho-c} becomes $\frac{e^{\pi\sqrt{3\lambda}}}{4\sqrt{3\lambda}}(1-\frac{1}{2\sqrt{3}C})(1+o(1))$ as $\lambda\to\infty$. Since we must have $\E[\CR(\rho,0)^\lambda]\to0$ as $\lambda\to\infty$, we find $C=\frac{\sqrt{3}}{6}$. Taking this into~\eqref{eq:cr-rho-c}, we conclude.
\end{proof}

As a corollary of Proposition~\ref{prop-cr-rho}, by the following special case of Tauberian theorem, we can derive the tail probabilities of $\CR(\rho,0)$ and $\CR(\wt\ell_1,0)$, which will be used in Section~\ref{section-backbone-sharp}.

\begin{lemma}\label{thm:tauberian}
Suppose $\mathcal{X}$ is a positive random variable such that $\E[1-e^{-\lambda \mathcal{X}}]=\frac{C}{|\log\lambda|}(1+o(1)),\ \lambda\downarrow0$ for some constant $C>0$. Then $\P[\mathcal{X}>t]=\frac{C}{\log t}(1+o(1))$ as $t\to\infty$.
\end{lemma}
\begin{proof}
Note that $\E[1-e^{-\lambda\mathcal{X}}]\ge(1-e^{-\lambda t})\P[\mathcal{X}>t]$ for $t>0$. Fix $M>0$ and take $\lambda=\frac{M}{t}$, by letting $t\to\infty$, then we find $\varlimsup\limits_{t\to\infty}\log t\cdot \P[\mathcal{X}>t]\le C(1-e^{-M})^{-1}$. Since $M>0$ is arbitrary, we have $\varlimsup\limits_{t\to\infty}\log t\cdot\P[\mathcal{X}>t]\le C$. On the other hand, note that for $t>1$,
\begin{equation}\label{eq:tauberian}
\log t\cdot\E[1-e^{-\lambda\mathcal{X}}]=\log t\cdot\int_0^\infty\lambda e^{-\lambda s}\P[\mathcal{X}>s]ds\le \log t\cdot(1-e^{-\lambda t})+\log t\cdot e^{-\lambda t}\P[\mathcal{X}>t]
\end{equation}
Choose $\alpha>1$. Taking $\lambda=t^{-\alpha}$ and then $t\to\infty$ in~\eqref{eq:tauberian} yields $\varliminf\limits_{t\to\infty}\log t\cdot\P[\mathcal{X}>t]\ge \alpha^{-1}C$. The lemma then follows by taking $\alpha\downarrow1$.
\end{proof}

\begin{corollary}
\label{cor-estimate-cr-rho}
We have $\P[\CR(\rho,0)<\varepsilon]=\frac{3}{\log|\log\varepsilon|}(1+o(1))$ and $\P[\CR(\wt\ell_1,0)<\varepsilon]=\frac{3}{\log|\log\varepsilon|}(1+o(1))$ as $\varepsilon\downarrow0$.
\end{corollary}
\begin{proof}
By Proposition~\ref{prop-cr-rho}, when $\lambda\in[0,\frac{1}{12})$, we have
\[
\E[\CR(\rho,0)^\lambda]=\frac{\sin(\frac{\pi}{2}\sqrt{1-12\lambda})}{\sqrt{1-12\lambda}}-\frac{2\sqrt{3}\sin(\frac{\pi}{3}\sqrt{1-12\lambda})}{\sqrt{1-12\lambda}\log\left(\frac{(2+\sqrt{3})^{2}-\tan^{2}\left(\frac{\pi}{12}\sqrt{1-12\lambda}\right)}{1-(2+\sqrt{3})^{2}\tan^{2}\left(\frac{\pi}{12}\sqrt{1-12\lambda}\right)}\right)}.
\]
In particular, $\E[\CR(\rho,0)^\lambda]=1-\frac{3}{|\log\lambda|}(1+o(1))$ as $\lambda\downarrow0$. By Lemma~\ref{thm:tauberian} (taking $\mathcal{X}$ there to be $-\log\CR(\rho,0)$), we find $\P[\CR(\rho,0)<\varepsilon]\sim\frac{3}{\log|\log\varepsilon|}$ as $\varepsilon\downarrow0$.

Similarly, combining Proposition~\ref{prop-cr-rho} and Lemma~\ref{lem-CR-BM-ell1}, for $\lambda\in[0,\frac{1}{12})$, we have
\[
\E[\CR(\wt\ell_1,0)^\lambda]=\E[\CR(\rho,0)^\lambda]\E[\CR(\ell_1,0)^\lambda]=1-\frac{2\sqrt{3}\sin(\frac{\pi}{3}\sqrt{1-12\lambda})}{\sin(\frac{\pi}{2}\sqrt{1-12\lambda})\log\left(\frac{(2+\sqrt{3})^{2}-\tan^{2}\left(\frac{\pi}{12}\sqrt{1-12\lambda}\right)}{1-(2+\sqrt{3})^{2}\tan^{2}\left(\frac{\pi}{12}\sqrt{1-12\lambda}\right)}\right)}.
\]
Hence $\E[\CR(\wt\ell_1,0)^\lambda]=1-\frac{3}{|\log\lambda|}(1+o(1))$ as $\lambda\downarrow0$, and we conclude by using Lemma~\ref{thm:tauberian}.
\end{proof}

\section{Proof of Theorem~\ref{thm:backbone}}
\label{section-backbone-sharp}
In this section, we prove Theorem~\ref{thm:backbone} based on the exact laws of the conformal radii in Section~\ref{section-backbone-1}. The general proof steps are akin to~\cite[proof of Lemma 2.6]{jego2023crossing}. Recall that $\mathsf{Bac}_\varepsilon$ is the event that there exist two disjoint subpaths on $B[0,\tau_\D]$ from $\varepsilon\S^1$ to $\frac{1}{2}\S^1$.
We first need the following continuous version of Menger's theorem on the connectedness of a graph.

\begin{theorem}[{\cite{nob32-n-arc,zippin-33-n,whyburn1948n}}]
\label{thm-n-arc}
Let $n\ge2$. Suppose $X$ is a connected, locally connected, and locally compact metric space such that $X$ cannot be disconnected by removing any $n-1$ points of $X$. Then for any two points $x,y\in X$, there exist $n$ arcs connecting $x$ and $y$, which are pairwise disjoint except on $x$ and $y$. Here an arc stands for a subset of $X$ that is homeomorphic to $[0,1]$.
\end{theorem}

\begin{lemma}\label{lem:mcmf}
    $\mathsf{Bac}_\varepsilon$ occurs
    if $\ell_n\cap\A_{1/2}\neq\emptyset$ and $\wt\ell_n\cap\varepsilon\D\neq\emptyset$ for some $n\ge1$.
\end{lemma}
\begin{proof}
Recall $s_n$ and $t_n$ in Definition~\ref{def:cut-point}.
Note that $B[s_n,t_n]$ forms a connected, locally connected, and locally compact metric space when equipped with the Euclidean metric on $\C$, and $B[s_n,t_n]$ cannot be disconnected by removing any single point (otherwise it would contradict the definition of $s_n$). Furthermore, on the event $\{\ell_n\cap\A_{1/2}\neq\emptyset\}\cap\{\wt\ell_n\cap\varepsilon\D\neq\emptyset\}$, there are $s,t\in[s_n,t_n]$ such that $B_s\in\varepsilon\D$ and $B_t\in\A_{1/2}$.
Thus by Theorem~\ref{thm-n-arc},
there exist two arcs on $B[s_n,t_n]$ from $B_s$ to $B_t$, staying disjoint apart from their endpoints. The result then follows by taking the sub-trajectories of these two arcs from their last hitting points of $\varepsilon\S^1$ to their first hitting points of $\frac{1}{2}\S^1$, respectively.
\end{proof}

We also need the following input.
\begin{lemma}\label{lem-finite-sum}
Define
\begin{equation}\label{eq:s}
S:=\sum_{n\ge1}\P[\ell_n\cap\A_{1/2}\neq\emptyset].
\end{equation}
Then $S\in(0,\infty)$.
\end{lemma}
\begin{proof}
For $n\ge2$, we have
\begin{align}
\P[\ell_n\cap\A_{1/2}\neq\emptyset]&\le\P[\CR(\ell_n,0)\ge e^{-(n-1)}]+\P[\ell_n\cap\A_{1/2}\neq\emptyset,\CR(\ell_n,0)<e^{-(n-1)}]\nonumber\\
&\le \P[n-1\ge \log\CR(\ell_n,0)^{-1}]+\P[\ell_n\cap\A_{1/2}\neq\emptyset,\dist(\ell_n,0)<e^{-(n-1)}]\nonumber\\
&\le \P\left[n-1\ge \sum_{i=1}^{n-1}\log\frac{\CR(\wt\ell_{i-1},0)}{\CR(\wt\ell_{i},0)}\right]+\P[\ell_n\cap\A_{1/2}\neq\emptyset,\dist(\ell_n,0)<e^{-(n-1)}].\label{eq:intersect}
\end{align}
By Lemma~\ref{lem-independent-BM-CR}, $\left(\log\frac{\CR(\wt\ell_{i-1},0)}{\CR(\wt\ell_{i},0)}\right)_{i\ge1}$ are i.i.d. random variables. Furthermore, Corollary~\ref{cor-estimate-cr-rho} yields that $\P\left[\log\frac{1}{\CR(\wt\ell_1,0)}>t\right]=\frac{3}{\log t}(1+o(1))$ as $t\to\infty$, hence $\E\left[\log\frac{1}{\CR(\wt\ell_1,0)}\right]=+\infty$. By the large deviation principle, there exists $u\in(0,1)$ such that 
\begin{equation}
\label{eq-lem-finite-sum-1}
\P\left[n-1\ge \sum_{i=1}^{n-1}\log\frac{\CR(\wt\ell_{i-1},0)}{\CR(\wt\ell_{i},0)}\right]\le u^{n-1}.
\end{equation}

We now bound the second term on the right side of~\eqref{eq:intersect} by the Brownian disconnection exponent derived in~\cite{lsw-bm-exponents2}. For $r\in(0,1)$, let $\tau_r=\inf\{t\ge0:|B_t|=r\}$ and $\sigma_r=\sup\{t\in[0,\tau_\D]:|B_t|=r\}$. Then for $s<r<1$, $(B_t)_{\sigma_s\le t\le\tau_r}$ has the same law as the Brownian excursion on annulus $\A_s\backslash\A_r$. By~\cite{lsw-bm-exponents2}, the probability that $(B_t)_{\sigma_s\le t\le\tau_r}$ does not disconnect $s\S^1$ and $r\S^1$ equals $(\frac{s}{r})^{1/4+o(1)}$ as $s\to0$ (we remark that such probability was exactly computed in the recent paper~\cite{cfsx}).
Since $(\ell_n\cap\A_{1/2}\neq\emptyset,\dist(\ell_n,0)<e^{-(n-1)})$ implies that $B[\sigma_{e^{-(n-1)}},\tau_{1/2}]$ does not disconnect $e^{-(n-1)}\S^1$ and $\frac{1}{2}\S^1$, we find
\begin{equation}
\label{eq-lem-finite-sum-2}
\P[\ell_n\cap\A_{1/2}\neq\emptyset,\dist(\ell_n,0)<e^{-(n-1)}]\le e^{-\frac{1}{4}(n-1)(1+o(1))}
\end{equation}
as $n\to\infty$.
The result then follows by combining~\eqref{eq:intersect},~\eqref{eq-lem-finite-sum-1} and~\eqref{eq-lem-finite-sum-2}.
\end{proof}

In the following, we will prove Theorem~\ref{thm:backbone} by showing that
\begin{equation}\label{eq:c-s}
\lim_{\varepsilon\to0}\P[\Bac_\varepsilon]\cdot(\log|\log\varepsilon|)=3S,
\end{equation}
where $S$ is defined in~\eqref{eq:s}. We first give the lower bound.
\begin{proof}[Proof of Theorem~\ref{thm:backbone}, the lower bound]
For $n\ge1$, let $A_{n,\varepsilon}=\{\ell_n\cap\A_{1/2}\neq\emptyset,\wt\ell_n\cap\varepsilon\D\neq\emptyset\}$, which satisfies $\cup_{n\ge1}A_{n,\varepsilon}\subset\Bac_\varepsilon$ by Lemma~\ref{lem:mcmf}. Furthermore, let $B_{n,\varepsilon}\subset A_{n,\varepsilon}$ be the event $\{\dist(\wt\ell_{n-1},0)>\varepsilon,\ell_n\cap\A_{1/2}\neq\emptyset,\frac{\CR(\wt\ell_{n},0)}{\CR({\ell_{n}},0)}\le \varepsilon\}$. Note that $B_{n,\varepsilon}\cap B_{m,\varepsilon}=\emptyset$ for every $n\neq m$. Thus, we have
\begin{equation}\label{eq:lower}
\P[\Bac_\varepsilon]\ge\sum_{n\ge1}\P[B_{n,\varepsilon}]=\P[\CR(\rho,0)<\varepsilon]\cdot\sum_{n\ge1}\P[\dist(\wt\ell_{n-1},0)>\varepsilon,\ell_n\cap\A_{1/2}\neq\emptyset],
\end{equation}
where we use Lemma~\ref{lem-independent-BM-CR} that $\frac{\CR(\wt\ell_{n},0)}{\CR({\ell_{n}},0)}$ is independent of $\ell_n$ and $\wt\ell_{n-1}$ in the second equality (recall that $\rho$ is the loop defined above Proposition~\ref{prop-cr-rho}).
Combining~\eqref{eq:lower} with Corollary~\ref{cor-estimate-cr-rho} and Lemma~\ref{lem-finite-sum}, we obtain that $\liminf_{\varepsilon\to0}\P[\Bac_\varepsilon]\cdot(\log|\log\varepsilon|)\ge3S$, as desired.
\end{proof}

Now we turn to the upper bound of Theorem~\ref{thm:backbone}, i.e.~$\limsup_{\varepsilon\to0}\P[\Bac_\varepsilon]\cdot(\log|\log\varepsilon|)\le3S$. This is based on the following basic fact.

\begin{lemma}
    \label{lem-variant-Koebe-1/4}
    Let $\varepsilon\in(0,10^{-2})$. Suppose $\ell\subset\D$ is a simple loop surrounding the origin with $\dist(\ell,0)>2\sqrt{\varepsilon}$, and let $f: D(\ell)\to \D$ be any conformal map such that $f(0)=0$. Then $f(\varepsilon\D)\subset2\sqrt{\varepsilon}\D$.
\end{lemma}
\begin{proof}
    By applying Schwartz's lemma for $f^{-1}$, we have $|f'(0)|^{-1}=\CR(\ell,0)\ge\dist(\ell,0)>2\sqrt{\varepsilon}$. Let $\widetilde{f}(z):=\frac{|f'(0)|}{2\sqrt{\varepsilon}}f^{-1}(2\sqrt{\varepsilon}z)$; note that $\widetilde{f}(0)=0$ and $|\widetilde{f}'(0)|=1$. By Koebe's 1/4 theorem, we have $\frac{1}{4}\D\subset\widetilde{f}(\D)$, which implies~$\frac{1}{4}\D\subset \frac{|f'(0)|}{2\sqrt{\varepsilon}}f^{-1}(2\sqrt{\varepsilon}\D)\subset \frac{1}{4\varepsilon}f^{-1}(2\sqrt{\varepsilon}\D)$, as desired.
\end{proof}

\begin{proof}[Proof of Theorem~\ref{thm:backbone}, the upper bound]
For $\varepsilon\in(0,10^{-10})$, let $\delta=2\sqrt{\varepsilon}$. Let 
$A_{n,\delta}=\{\ell_n\cap\A_{1/2}\neq\emptyset,\wt\ell_n\cap\delta\D\neq\emptyset\}$ be as above. Then by the definition of $A_{1,\delta}$, we have
\begin{align}
\P[\Bac_\varepsilon]&\le\P[\Bac_\varepsilon\backslash \cup_{n\ge1}A_{n,\delta}]+\P[\cup_{n\ge1}A_{n,\delta}]\nonumber\\
&\le \P[\Bac_\varepsilon,\ell_1\cap\A_{1/2}=\emptyset]+\P[\Bac_\varepsilon\backslash\cup_{n\ge2}A_{n,\delta},\ell_1\cap\A_{1/2}\neq\emptyset,\wt\ell_1\cap\delta\D=\emptyset]+\P[\cup_{n\ge1}A_{n,\delta}]\nonumber\\
&\le\P[\ell_1\cap\varepsilon\D\neq\emptyset]+\P[\Bac_\varepsilon\backslash \cup_{n\ge2}A_{n,\delta},\wt\ell_1\cap\A_{1/2}\neq\emptyset,\wt\ell_1\cap\delta\D=\emptyset]+\P[\cup_{n\ge1}A_{n,\delta}].\label{eq-proof-lower-0a}
\end{align}
Here the last inequality uses the fact that $\{\Bac_\varepsilon,\ell_1\cap\A_{1/2}=\emptyset\}$ implies $\ell_1\cap\varepsilon\D\neq\emptyset$, while $\{\Bac_\varepsilon\backslash\cup_{n\ge2}A_{n,\delta},\ell_1\cap\A_{1/2}\neq\emptyset,\wt\ell_1\cap\delta\D=\emptyset\}$ implies $\{\Bac_\varepsilon\backslash \cup_{n\ge2}A_{n,\delta},\wt\ell_1\cap\A_{1/2}\neq\emptyset,\wt\ell_1\cap\delta\D=\emptyset\}$. Similarly, by~\eqref{eq-proof-lower-0a} and using the definition of $A_{n,\delta}$ iteratively, we have
\begin{equation*}
\begin{aligned}
\P[\Bac_\varepsilon]&\le\P[\ell_1\cap\varepsilon\D\neq\emptyset]+\P[\wt\ell_1\cap\A_{1/2}\neq\emptyset,\wt\ell_1\cap\delta\D=\emptyset,\ell_2\cap\varepsilon\D\neq\emptyset]\\
&+\P[\Bac_\varepsilon\backslash \cup_{n\ge3}A_{n,\delta},\wt\ell_2\cap\A_{1/2}\neq\emptyset,\wt\ell_2\cap\delta\D=\emptyset]+\P[\cup_{n\ge1}A_{n,\delta}]\\
&\le...\\
&\le\P[\ell_1\cap\varepsilon\D\neq\emptyset]+\sum_{n=1}^m\P[\wt\ell_n\cap\A_{1/2}\neq\emptyset,\wt\ell_n\cap\delta\D=\emptyset,\ell_{n+1}\cap\varepsilon\D\neq\emptyset]\\
&+\P[\Bac_\varepsilon\backslash \cup_{n\ge m+2}A_{n,\delta},\wt\ell_{m+1}\cap\A_{1/2}\neq\emptyset,\wt\ell_{m+1}\cap\delta\D=\emptyset]+\P[\cup_{n\ge1}A_{n,\delta}]
\end{aligned}
\end{equation*}
for any $m\ge1$. Taking $m\to\infty$ and using that $\lim\limits_{n\to\infty}\P[\wt\ell_n\cap\A_{1/2}\neq\emptyset]=0$, we obtain that
\begin{equation}
\label{eq-proof-lower-0b}
\P[\Bac_\varepsilon] \le\P[\ell_1\cap\varepsilon\D\neq\emptyset]+\sum_{n\ge1}\P[\wt\ell_n\cap\A_{1/2}\neq\emptyset,\wt\ell_n\cap\delta\D=\emptyset,\ell_{n+1}\cap\varepsilon\D\neq\emptyset]+\P[\cup_{n\ge1}A_{n,\delta}].
\end{equation}

It suffices to bound each of the three terms on the right side of~\eqref{eq-proof-lower-0b}. First, by Lemma~\ref{lem-CR-BM-ell1} and Koebe's 1/4 theorem, there exists $C_1\in(0,\infty)$ such that
\begin{equation}
\label{eq-proof-lower-1}
\P[\ell_1\cap\varepsilon\D\neq\emptyset]\le C_1\varepsilon^{\frac{1}{4}}.
\end{equation}

Recall that $\wt g_n$ is the conformal map from $D(\wt\ell_n)$ to $\D$ that fixes the origin with $\wt g_n(B_{s_n})=1$. By Lemma~\ref{lem-variant-Koebe-1/4}, $\{\wt\ell_n\cap\delta\D=\emptyset\}$ implies $\{\wt g_n(\varepsilon\D)\subset \delta\D\}$.
Furthermore, by Lemma~\ref{lem-independent-BM-CR}, we know that $\wt g_n(\ell_{n+1})$ is independent of $\wt\ell_n$, and $\P[\wt g_n(\ell_{n+1})\cap\delta\D\neq\emptyset]=\P[\ell_1\cap\delta\D\neq\emptyset]$. Hence, we have
\begin{align}
\sum_{n\ge1}\P[\wt\ell_n\cap\A_{1/2}\neq\emptyset,\wt\ell_n\cap\delta\D=\emptyset,\ell_{n+1}\cap\varepsilon\D\neq\emptyset]&\le \sum_{n\ge1}\P[\wt\ell_n\cap\A_{1/2}\neq\emptyset,\wt\ell_n\cap\delta\D=\emptyset,\wt g_n(\ell_{n+1})\cap\delta\D\neq\emptyset]\nonumber\\
&\le\P[\ell_1\cap\delta\D\neq\emptyset]\cdot\sum_{n\ge1}\P[\wt\ell_n\cap\A_{1/2}\neq\emptyset]\nonumber\\
&\le 2C_1S\varepsilon^{\frac{1}{8}},\label{eq-proof-lower-2}
\end{align}
where the last line uses~\eqref{eq-proof-lower-1} and $S$ is defined in~\eqref{eq:s}.

Now let $N_\delta=\inf\{n\ge1:A_{n,\delta}\text{ occur}\}$, and hence $\{\cup_{n\ge1}A_{n,\delta}\}=\{N_\delta<\infty\}$. By Koebe's 1/4 theorem, we have
\begin{align}
\P[\cup_{n\ge1}A_{n,\delta}]&=\P[N_\delta<\infty]=\P[N_\delta<\infty,\CR(\wt\ell_{N_\delta},0)\le 4\delta]\nonumber\\
&\le\P\left[N_\delta<\infty,\frac{\CR(\wt\ell_{N_\delta},0)}{\CR(\ell_{N_\delta},0)}<2\sqrt{\delta}\right]+\P[N_\delta<\infty,\CR(\ell_{N_\delta},0)<2\sqrt{\delta}].\label{eq-proof-lower-3}
\end{align}
Note that the definition of $N_\delta$ implies
\begin{align*}
\P\left[N_\delta<\infty,\frac{\CR(\wt\ell_{N_\delta},0)}{\CR(\ell_{N_\delta},0)}<2\sqrt{\delta}\right]&=\sum\limits_{n=1}^\infty\P\left[N_\delta=n,\frac{\CR(\wt\ell_n,0)}{\CR(\ell_n,0)}<2\sqrt{\delta}\right]\\
&\le \sum\limits_{n=1}^\infty\P\left[\ell_n\cap\A_{1/2}\neq\emptyset,\frac{\CR(\wt\ell_n,0)}{\CR(\ell_n,0)}<2\sqrt{\delta}\right]\\
&=\P[\CR(\rho,0)<2\sqrt{\delta}]\cdot S,
\end{align*}
where the last line follows from Lemma~\ref{lem-independent-BM-CR}. Meanwhile, $\{N_\delta<\infty,\CR(\ell_{N_\delta},0)<2\sqrt{\delta}\}$ implies $\{N_\delta<\infty, \ell_{N_\delta}\cap\A_{1/2}\neq\emptyset,\ell_{N_\delta}\cap2\sqrt{\delta}\D\neq\emptyset\}$, which further implies the Brownian trajectory $B[\sigma_{2\sqrt{\delta}},\tau_{1/2}]$ does not disconnect $2\sqrt{\delta}\S^1$ and $\frac{1}{2}\S^1$. Combined with~\eqref{eq-proof-lower-3} and using again the Brownian disconnection exponent $\frac{1}{4}$ from~\cite{lsw-bm-exponents2}, we find
\begin{align}
\P[\cup_{n\ge1}A_{n,\delta}]
&\le \P[\CR(\rho,0)<2\sqrt{\delta}]\cdot S+\P\left[B[\sigma_{2\sqrt{\delta}},\tau_{1/2}]\text{ does not disconnect }2\sqrt{\delta}\S^1\text{ and }\frac{1}{2}\S^1\right]\nonumber\\
&\le \P[\CR(\rho,0)<2\sqrt{\delta}]\cdot S+(4\sqrt{\delta})^{\frac{1}{4}+o(1)}.\label{eq-proof-lower-4}
\end{align}

Finally, by combining~\eqref{eq-proof-lower-0b},~\eqref{eq-proof-lower-1},~\eqref{eq-proof-lower-2},~\eqref{eq-proof-lower-4}, we have (recall that $\delta=2\sqrt{\varepsilon}$)
$$
\P[\Bac_\varepsilon]\le\P[\CR(\rho,0)<2\sqrt{\delta}]\cdot S+(4\sqrt{\delta})^{\frac{1}{4}+o(1)}+C_1\varepsilon^{\frac{1}{4}}+2C_1S\varepsilon^{\frac{1}{8}}.
$$
Using Corollary~\ref{cor-estimate-cr-rho}, we conclude that $\limsup_{\varepsilon\to0}\P[\Bac_\varepsilon]\cdot(\log|\log\varepsilon|)\le3S$, as desired.
\end{proof}

\appendix
\section{Proof of Theorem~\ref{thm-cr-mod-relation}}
\label{appendix}

In this appendix, we provide the proof of Theorem~\ref{thm-cr-mod-relation}, which is implicit in~\cite{ARS-Annulus}. For convenience, we consider the upper half plane $\hH:=\{z\in\C:\text{Im}(z)>0\}$, and let $\sm$ be the restriction of the $\SLE_{8/3}$ loop measure on $\C$ to the loops that are contained in $\hH$ and surround $i$. According to the conformal restriction~\cite{werner-loops}, if $\Phi:\D\to\hH$ is a conformal map with $\Phi(0)=i$, then $\sm=\Phi_*\SLE_{8/3,\D}^{\lp}$. Therefore, to prove Theorem~\ref{thm-cr-mod-relation}, it suffices to show the following
\begin{theorem}
\label{thm-appendix}
There is a constant $C>0$ such that for any measurable function $f:\R_+\to\R_+$ and $\lambda\ge0$, we have
$$
\int f\left(\Mod(\hH\backslash\overline{D(\eta)})\right)\psi_\eta'(i)^\lambda\sm(d\eta)=
C\frac{\sqrt{12\lambda-1}}{\sinh\left(\frac{\pi}{3}\sqrt{12\lambda-1}\right)}\int_0^\infty e^{-(2\lambda-\frac{1}{6})\pi\tau}\eta(2i\tau)f(\tau)d \tau.
$$
Here $\psi_\eta$ is the conformal map from $\hH$ to $D(\eta)$ such that $\psi_\eta(i)=i$ and $\psi_\eta'(i)>0$.
\end{theorem}

The proof of Theorem~\ref{thm-appendix} is based on the SLE-coupled Liouville quantum gravity with parameter $\sqrt{\frac{8}{3}}$. In the following we fix $\gamma=\sqrt{\frac{8}{3}}$ and let $Q=\frac{\gamma}{2}+\frac{2}{\gamma}$. Let $\P_{\hH}$ be the law of the free boundary Gaussian free field (GFF) on $\hH$ with mean zero on the upper semi-circle $\S^1\cap\hH$. 

\begin{definition}[Liouville field on $\hH$]
\label{def-LF-H}
Let $(h,c)$ be sampled from $\P_{\hH}\times[e^{-Qc}dc]$, and let $\phi(z)=h(z)-2Q\log|z|_++c$ (here $|z|_+:=\max\{|z|,1\}$). Denote the law of $\phi$ by $\LF_{\hH}$.

For $\alpha\in\R$, let $\LF_{\hH}^{(\alpha,i)}(d\phi):=\lim_{\varepsilon\to0}\varepsilon^{\frac{\alpha^2}{2}}e^{\alpha\phi_\varepsilon(i)}\LF_{\hH}(d\phi)$ (the limit exists in the vague topology, see~\cite[Lemma 2.2]{ARS-FZZ}). Here $\phi_\varepsilon(i)$ is the average of $\phi$ over the circle $\partial B_\varepsilon(i)$.
\end{definition}

For a sample $\phi$ from $\LF_{\hH}^{(\alpha,i)}$, we can define its boundary Gaussian multiplicative chaos (GMC) measure $\nu_{\phi}:=\lim\limits_{\varepsilon\to0}\varepsilon^{\frac{\gamma^2}{4}}e^{\frac{\gamma}{2}\phi_\varepsilon(x)}dx$, where $\phi_\varepsilon(x)$ is the average of $\phi$ over the semi-circle $\partial B_\varepsilon(x)\cap\hH$.
Then denote $\{\LF_{\hH}^{(\alpha,i)}(\ell)\}_{\ell>0}$ to be the disintegration of $\LF_{\hH}^{(\alpha,i)}$ over $\nu_\phi(\R)$; i.e.~$\LF_{\hH}^{(\alpha,i)}=\int_0^\infty\LF_{\hH}^{(\alpha,i)}(\ell)d\ell$. We refer readers to see e.g.~\cite[Section 2.2]{ARS-FZZ} for further details. In particular, by~\cite[Lemma 2.7]{ARS-FZZ}, for $\alpha>\frac{\gamma}{2}$,
there are constants $C_{\alpha,\gamma}>0$ such that $|\LF_{\hH}^{(\alpha,i)}(\ell)|=C_{\alpha,\gamma}\ell^{\frac{2}{\gamma}(\alpha-Q)-1}$ for any $\ell>0$.

We also need the Liouville field on the annulus. For $\tau>0$, let $\mathcal{C}_\tau=[0,\tau]\times[0,1]/\mathord\sim$ be the horizontal cylinder with modulus $\tau$, where under $\sim$ we identify $(x,0)$ and $(x,1)$ for each $x\in[0,\tau]$. Let $\P_{\tau}$ be the free boundary GFF on $\mathcal{C}_\tau$
with zero mean on the circle $\{\frac{\tau}{2}\}\times[0,1]/\mathord\sim$.
\begin{definition}[{\cite[Definition 2.2]{ARS-Annulus}}]
\label{def-LF-annulus}
Let $(h,c)$ be sampled from $\P_\tau\times dc$. Then denote the law of $\phi=h+c$ by $\LF_\tau$ (which is an infinite measure on $H^{-1}(\mathcal{C}_\tau)$).
\end{definition}

For a sample $\phi$ from $\LF_\tau$,
we can similarly define its boundary GMC measures $\nu^1_\phi, \nu_\phi^2$ on the two boundaries $\partial_1\mathcal{C}_\tau=\{0\}\times[0,1]\mathord\sim$ and $\partial_2\mathcal{C}_\tau=\{1\}\times[0,1]\mathord\sim$, respectively. We also let $\{\LF_\tau(\ell_1,\ell_2)\}_{\ell_1,\ell_2>0}$ be the disintegration of $\LF_\tau$ over $\nu_\phi^1(\partial\cC_1)$ and $\nu_\phi^2(\partial\cC_2)$, i.e.~$\LF_\tau=\iint_{\R_+^2}\LF_\tau(\ell_1,\ell_2)d\ell_1d\ell_2$. The following result from~\cite{ARS-Annulus} gives the exact solvability of $|\LF_\tau(\ell_1,\ell_2)|$, which is based on~\cite{wu-annulus}.

\begin{proposition}[{\cite[Equation (3.6)]{ARS-Annulus}}]
 \label{prop-LF-integrability}
For $\tau>0$ and $y\in(-1,\frac{4}{\gamma^2})$, we have
$$
\iint_{\R_+^2}\ell_1e^{-\ell_1}\ell_2^{y}|\LF_\tau(\ell_1,\ell_2)|d\ell_1d\ell_2=\frac{\pi\gamma y\Gamma(1+y)}{2\sin(\frac{\gamma^2}{4}\pi y)}e^{\frac{\pi}{4}\gamma^2\tau y^2}.
$$
\end{proposition}

For two domains $D,\wt{D}\subset\C$ and a conformal map $g:D\to\wt{D}$, when $h$ is a distribution on $D$, we define $g\bullet h:=h\circ g^{-1}+Q\log|(g^{-1})'|$ (which is a distribution on $\wt D$). Let $\Omega$ be the space of simple loops in $\hH$ surrounding $i$, and $\Conf(\hH,i)$ be the group of conformal automorphisms of $\hH$ that fix $i$. For $(\phi,\eta,g,\theta)\in H^{-1}(\hH)\times\Omega\times \Conf(\hH,i)\times[0,1]$, define a measurable map $F$ by
$$
F(\phi,\eta,g,\theta):=\left((g\circ \psi_\eta^{-1})\bullet \phi|_{D(\eta)}, f_{\eta}^\theta\bullet \phi|_{\hH\backslash\overline{D(\eta)}},\Mod(\hH\backslash\overline{D(\eta)})\right).
$$
Here $\psi_\eta:\hH\to D(\eta)$ is the conformal map such that $\psi_\eta(i)=i$ and $\psi_\eta'(i)>0$, and
$f_{\eta}^\theta:\hH\backslash\overline{D(\eta)}\to\mathcal{C}_\tau$ is the conformal map such that $f_{\eta}^\theta(0)=\theta i$ with $\tau=\Mod(\hH\backslash\overline{D(\eta)})$.

Let $\Haar_{(\hH,i)}$ be the Haar measure on $\Conf(\hH,i)$ such that $|\Haar_{(\hH,i)}|=1$, and let $\Unif_{[0,1]}$ be the uniform measure on $[0,1]$.
The following proposition, which is essentially from~\cite{ARS-Annulus}, describes the law of $\LF_\hH^{(\gamma,i)}$ when cut by a simple loop sampled from $\sm$.

\begin{proposition}\label{prop-sle-weld}
There is a constant $C>0$ such that for any $\ell_1>0$,
$$
F_*\left(\LF_{\hH}^{(\gamma,i)}(\ell_1)\times\sm\times\Haar_{(\hH,i)}\times\Unif_{[0,1]}\right)
=C\int_0^\infty \left(\int_0^\infty \ell_2\LF_{\hH}^{(\gamma,i)}(\ell_2)(d\phi_1)\times\LF_\tau(\ell_2,\ell_1)(d\phi_2)d\ell_2\right)
\eta(2i\tau)d\tau.
$$
Here $F_*$ stands for the pushforward of measures, and we view the right side as a measure on $(\phi_1,\phi_2,\tau)$.
\end{proposition}
\begin{proof}
By~\cite{ARS-Annulus}, the $\SLE_{8/3}$ loop cut the Brownian disk into an independent pair of a (smaller) Brownian disk and a Brownian annulus; see~\cite[Proposition 4.5]{cfsx} for the precise statement and proof. The result then follows from that the uniform embedding of the Brownian disk on $\hH$ gives the Liouville field on $\hH$~\cite[Theorem 3.4]{ARS-FZZ}.
\end{proof}

For $\alpha\in\R$, let $\Delta_\alpha=\frac{\alpha}{2}(Q-\frac{\alpha}{2})$. Define the measure $\sm^\alpha$ via $\frac{d\sm^\alpha}{d\sm}(\eta)=\psi_\eta'(i)^{2\Delta_\alpha-2}$, where $\psi_\eta:\hH\to D(\eta)$ is the conformal map as above. The following proposition is obtained from Proposition~\ref{prop-sle-weld} by a standard reweighting argument. Such an argument was first developed in~\cite{AHS-SLE-integrability}, and appeared in many recent papers on the integrability of SLE/CLE, see e.g.~\cite{ARS-FZZ,ACSW24b,ARS-Annulus,nolin2024backboneexponenttwodimensionalpercolation}. 

\begin{proposition}
 \label{prop-sle-reweight}
There is a constant $C>0$ such that for any $\ell_1>0$ and $\alpha\in\R$,
$$
F_*\left(\LF_{\hH}^{(\alpha,i)}(\ell_1)\times\sm^\alpha\times\Haar_{(\hH,i)}\times\Unif_{[0,1]}\right)
=C\int_0^\infty \left(\int_0^\infty \ell_2\LF_{\hH}^{(\alpha,i)}(\ell_2)(d\phi_1)\times\LF_\tau(\ell_2,\ell_1)(d\phi_2)d\ell_2\right)
\eta(2i\tau)d\tau.
$$
As in Proposition~\ref{prop-sle-weld}, we view the right side above as a measure on $(\phi_1,\phi_2,\tau)$.
\end{proposition}

\begin{proof}
For $\phi\in H^{-1}(\hH)$, $\eta\in\Omega$ and $g\in\Conf(\hH,i)$, let $\phi_1=(g\circ \psi_\eta^{-1})\bullet \phi|_{D(\eta)}$. Recall that for $\varepsilon>0$, $(\phi_1)_\varepsilon(i)$ is the average of $\phi_1$ over $\partial B_\varepsilon(i)$.
By~\cite[Lemma 4.8]{ARS-FZZ},
as $\varepsilon\to0$, we have
$$
\begin{aligned}
&F_*\left(\varepsilon^{\frac{1}{2}(\alpha^2-\gamma^2)}e^{(\alpha-\gamma)(\phi_1)_\varepsilon(i)}\LF_{\hH}^{(\gamma,i)}(\ell_2)(d\phi)\times\sm(d\eta)\times\Haar_{(\hH,i)}\times\Unif_{[0,1]}\right)\\
=~&F_*\left(\varepsilon^{\frac{1}{2}(\alpha^2-\gamma^2)}e^{(\alpha-\gamma)(\phi\circ\psi_\eta\circ g^{-1})_\varepsilon(i)}\LF_{\hH}^{(\gamma,i)}(\ell_2)(d\phi)\times|\psi_\eta'(i)|^{(\alpha-\gamma)Q}\sm(d\eta)\times\Haar_{(\hH,i)}\times\Unif_{[0,1]}\right)\\
\to~& F_*\left(\LF_{\hH}^{(\alpha,i)}(\ell_2)\times\sm^\alpha\times\Haar_{(\hH,i)}\times\Unif_{[0,1]}\right).
\end{aligned}
$$
On the other hand, by Proposition~\ref{prop-sle-weld} and~\cite[Lemma 4.7]{ARS-FZZ}, we know that
$$
\begin{aligned}
&F_*\left(\varepsilon^{\frac{1}{2}(\alpha^2-\gamma^2)}e^{(\alpha-\gamma)(\phi_1)_\varepsilon(i)}\LF_{\hH}^{(\gamma,i)}(\ell_1)(d\phi)\times\sm(d\eta)\times\Haar_{(\hH,i)}\times\Unif_{[0,1]}\right)\\
=~&C\int_0^\infty \left(\int_0^\infty \varepsilon^{\frac{1}{2}(\alpha^2-\gamma^2)}e^{(\alpha-\gamma)(\phi_1)_\varepsilon(i)}\ell_2\LF_{\hH}^{(\gamma,i)}(\ell_2)(d\phi_1)\times\LF_\tau(\ell_2,\ell_1)(d\phi_2)d\ell_2\right)
\eta(2i\tau)d\tau\\
\to~& C\int_0^\infty \left(\int_0^\infty \ell_2\LF_{\hH}^{(\alpha,i)}(\ell_2)(d\phi_1)\times\LF_\tau(\ell_2,\ell_1)(d\phi_2)d\ell_2\right)
\eta(2i\tau)d\tau
\end{aligned}
$$
as $\varepsilon\to0$. Combined, we conclude.
\end{proof}

Now we finish the proof of Theorem~\ref{thm-appendix}.

\begin{proof}[Proof of Theorem~\ref{thm-appendix}]
Let $f:\R_+\to\R_+$ be a measurable function. From Proposition~\ref{prop-sle-reweight}, we have
\begin{equation}
\label{eq-appendix-1}
\begin{aligned}
& F_*\left(\LF_{\hH}^{(\alpha,i)}(\ell_1)\times f(\Mod(\hH\backslash\overline{D(\eta)}))\sm^\alpha(d\eta)\times\Haar_{(\hH,i)}\times\Unif_{[0,1]}\right)\\
=~&C\int_0^\infty \left(\int_0^\infty \ell_2\LF_{\hH}^{(\alpha,i)}(\ell_2)(d\phi_1)\times\LF_\tau(\ell_2,\ell_1)(d\phi_2)d\ell_2\right)
\eta(2i\tau)f(\tau)d\tau.
\end{aligned}
\end{equation}
Taking the total mass on both sides of~\eqref{eq-appendix-1}, we obtain that
\begin{equation}
\label{eq-appendix-2}
\ell_1^{\frac{2}{\gamma}(\alpha-Q)-1}\int f\left(\Mod(\hH\backslash\overline{D(\eta)})\right)|\psi_\eta'(i)|^{2\Delta_\alpha-2}\sm(d\eta)=C\iint_{\R_+^2}\ell_2^{\frac{2}{\gamma}(\alpha-Q)}|\LF_\tau(\ell_2,\ell_1)|\eta(2i\tau)f(\tau)d\ell_2d\tau;
\end{equation}
here we use that $|\LF_{\hH}^{(\alpha,i)}(\ell)|=C_{\alpha,\gamma}\ell^{\frac{2}{\gamma}(\alpha-Q)-1}$ for $\alpha>\frac{\gamma}{2}$.
Then integrating both sides of~\eqref{eq-appendix-2} with respect to $\1_{\ell_1>0}\ell_1e^{-\ell_1}d\ell_1$, by using Proposition~\ref{prop-LF-integrability}, we find
\begin{equation*}
\begin{aligned}
&\Gamma\left(\frac{2}{\gamma}(\alpha-Q)+1\right)\int f\left(\Mod(\hH\backslash\overline{D(\eta)})\right)|\psi_\eta'(i)|^{2\Delta_\alpha-2}\sm(d\eta)\\
=~&C\iiint_{\R_+^3}\ell_2^{\frac{2}{\gamma}(\alpha-Q)}\ell_1e^{-\ell_1}|\LF_\tau(\ell_2,\ell_1)|d\ell_1d\ell_2\eta(2i\tau)f(\tau)d\tau\\
=~&C'\Gamma\left(\frac{2}{\gamma}(\alpha-Q)+1\right)\frac{(\alpha-Q)}{\sin(\frac{\gamma}{2}\pi(\alpha-Q))}
\int_0^\infty e^{(\alpha-Q)^2\pi\tau}\eta(2i\tau)f(\tau)d\tau
\end{aligned}
\end{equation*}
for some constant $C'>0$. 
Then we conclude by taking $\lambda=2\Delta_\alpha-2$. 
\end{proof}

\def\cprime{$'$}

\end{document}